\documentclass{article}
\usepackage{amsmath}
\usepackage{amssymb}
\usepackage{amsfonts}
\usepackage{amsthm}
\usepackage{ascmac}
\usepackage{authblk}

\title{An Approach to Non-Abelian Cyclotomic Fields}
\author{Shinji Ishida}

\providecommand{\keywords}[1]{{\textit{Keywords:}} #1}

\makeatletter
\@addtoreset{equation}{section}

\newtheorem{defi}{Definition}[section]
\newtheorem{thm}{Theorem}[section]
\newtheorem{prop}{Proposition}[section]
\newtheorem{lemm}{Lemma}[section]
\newtheorem{coro}{Corollary}[section]
\newtheorem{exam}{Example}[section]
\newtheorem{conji}{Conjecture}[section]

\begin{document}
\maketitle
\begin{abstract}
We mainly study a polynomial $f_{1,n}(x)=x^{n-1} + 2x^{n-2} + 3x^{n-3} + \cdots + kx^{n-k} + \cdots + (n-1)x + n$ over $\mathbb{Z}$ and the Galois group of the minimal splitting field. First, we show that an arbitrary root $\alpha_{n}$ of $f_{1,n}(x)$ satisfies $|\alpha_{n}|\to 1$ ($n\to \infty$), and discuss the irreducibility of $f_{1,n}(x)$ over $\mathbb{Z}$ for several type $n$. After that, we show that the Galois group of $f_{1,n}(x)$ is Symmetric group $S_{n-1}$ for several type $n$. Although those roots of $f_{1,n}(x)=0$ don't draw an exact circle, it looks like a circle on complex plane. Moreover by considering that Galois groups of $f_{1,n}(x)$ are not abelian in many cases, we call such extension fields over $\mathbb{Q}$ ``Non-Abelian Cycrotomic Fields'' here.
\end{abstract}

\keywords\footnote[1]{The subject classification codes by 2010 Mathematics Subject Classification is primary-11R21 and secondary 11R37.}{Algebraic number field, Non-Abelian extension field over rational field $\mathbb{Q}$, cyclotomic field, and polynomials over rational integer $\mathbb{Z}$.}

\tableofcontents
\section{Introduction}
In this section, first, we explain several fundamental properties of the following polynomial over $\mathbb{Z}$
\begin{equation}
f_{1,n}(x)=x^{n-1} + 2x^{n-2} + 3x^{n-3} + \cdots + kx^{n-k} + \cdots + (n-1)x + n.
\end{equation}
Second, we explain the reason why we call such extension fields defined by $f_{1,n}(x)$ over $\mathbb{Q}$ ``non-Abelian cyclotomic field'' when the Galois group is not abelian. After that, we explain the reason why we studied such polynomials.
\\

\subsection{The fundamental properties}
The fundamental properties of $f_{1,n}(x)$ are as follows. As for the proofs, look at $\S 2$ and $\S 3$.
\begin{enumerate}
	\item[P1.]$\textless${\bf The distribution of the roots}$\textgreater$\\
	Let $\alpha_{n}$ be an arbitrary root of $f_{1,n}(x)=0$. Then $1\leq |\alpha_{n}|<(n+1)^{2/n}$. In particular, $|\alpha_{n}|\to 1$ ($n\to \infty$).
	\\
	\item[P2.]$\textless${\bf The irreducibility}$\textgreater$\\
	If $n+1$ is a prime number, or $n$ is a power of prime number, then $f_{1,n}(x)$ is irreducible over $\mathbb{Z}$. By using "SageMath", we can know that $f_{1,n}(x)$ is irreducible for all $1< n \leq 3000$. {\bf Hence we conjecture that $f_{1,n}(x)$ is irreducible over $\mathbb{Z}$ for all $n > 1$}.
	\\
	\item[P3.]$\textless${\bf The discriminant and the subfield of the minimal splitting field}$\textgreater$\\
	The discriminant $D(f_{1,n}(x))$ of $f_{1,n}(x)$ is $(-1)^{\frac{(n-1)(n+2)}{2}}2n^{n-3}(n+1)^{n-2}$, and it is not a square of a national number in case of $n\equiv 0$ or $3$ (mod 4) because it is negative.  For $n>2$, let $K_{n}$ be the minimal splitting field of $f_{1,n}(x)$ over $\mathbb{Q}$. Then $K_{n}$ has the following quadratic field $F$ as a sub field:
\begin{eqnarray}
F=\left\{ \begin{array}{ll}
\mathbb{Q}(\sqrt{-2l}) & (n=4l) \\
\mathbb{Q}(\sqrt{2l+1})  & (n=4l+1) \\
\mathbb{Q}(\sqrt{2l+1})  & (n=4l+2) \\
\mathbb{Q}(\sqrt{-2(l+1)})  & (n=4l+3) \\
\end{array} \right.
\end{eqnarray}
\end{enumerate}
where, we must assume that $2l+1$ is a square free rational integer for $n=4l+1$ and $4l+2$. 
\\

\begin{exam}
\upshape
For $n=4$, Galois group $G_{4}$ of the minimal splitting field $K_{4}/\mathbb{Q}$ of $f_{1,4}(x) = x^3 + 2x^2 + 3x + 4$ is Symmetric group $S_{3}$.
\end{exam}
\begin{proof}
Actually, $f_{1,4}(x)$ is irreducible by the property P2, and $D(f_{1,4}(x)) = -200 = -2^{3}5^{2}$. Hence the splitting field $K_{4}$ of  $f_{1,4}(x)$ has a sub quadratic field $\mathbb{Q}(\sqrt{-2})$. This indicates $G_{4}\simeq  S_{3}$.
\end{proof}

\begin{exam}
\upshape
For $n=5$, Galois group $G_{5}$ of the minimal splitting field $K_{5}/\mathbb{Q}$ of $f_{1,5}(x) = x^4 + 2x^3 + 3x^2 + 4x + 5$ is Symmetric group $S_{4}$.
\end{exam}

\begin{proof}
By the property P2, $f_{1,5}(x)$ is irreducible over $\mathbb{Z}$ and $G_{5}$ has a transposition by the property P3 and Chebotarev's density theorem. Moreover, $f_{1,5}(x)$ is irreducible (mod 7) and $f_{1,5}(x)\equiv (x-7)(x^3 + 9x^2 + 4)$ (mod 11). This indicates that $G_{5}$ has two elements of order 4 and 3 respectively. Hence $G_{5}\simeq S_{4}$. 
\end{proof}

More generally, we have the following main result.
\begin{thm}
\upshape
Assume that  $f_{1,n}(x)$ is irreducible over $\mathbb{Z}$ and $p = n-1$ is an odd prime number, and let $G_{n}$ be the Galois group of the minimal splitting field $K_{n}$ of $f_{1,n}(x)$ over $\mathbb{Q}$. Then we obtain the following result.
\begin{enumerate}
	\item[(1)] $G_{n}$ is not abelian.
	\item[(2)] If $\sqrt{D(f_{1,n}(x))}\notin\mathbb{Z}$, $G_{n}$ is Symmetric group $S_{n-1}$.
\end{enumerate}
\end{thm}

\begin{proof}
(1) $G_{n}$ has an element $\sigma$ of order 2 which is a restriction of complex conjugacy, and an element $\rho$ of order $p$($p$-cyclic) because $p = n-1$ is an odd prime and $G_{n}$ acts transitively the of all roots of $f_{1,n}(x) = 0$ (from those conditions, $p$ divides the order of $G_{n}$ and this indicates $G_{n}$ has an element of order $p$). It is easy to see that $f_{1,n}(x) = 0$ has an only one real root $\alpha$ and $\sigma \circ \rho (\alpha) \neq \rho \circ \sigma (\alpha)$.
\\

(2) This is proved by using Chebotarev's density theorem and a theorem of Group Theory (see \cite{Fernando} [Theorem 4.12.8]): {\it``Let p be a prime number. If $G\subset S_{p}$ is transitive and contains a transposition, then $G=S_{p}$}''. Because $\sqrt{D(f_{1,n}(x))}\notin\mathbb{Z}$, the minimal splitting field $K_{n}$ over $\mathbb{Q}$ of $f_{1,n}(x)$ has a quadratic field $F$ over $\mathbb{Q}$. Hence by Chebotarev's density theorem, there are an infinite numbers of rational prime numbers $q$ such that unramified and $(q)$ is a prime ideal in $O_{F}$ (integer ring of $F$), and splits completely in $K_{n}$. Hence for such a prime number $q$, $f_{1,n}(x)$ (mod $q$) $=$ (irreducible quadratic equation) $\ast \Pi$ (distinct linear equations). This indicates that Galois group $G_{n}$ has a transposition. Moreover, Galois group acts transitively on the set of all roots generally, we obtain Theorem 1.1. (2).
\end{proof}

As such examples, we can consider some Mersenne prime numbers as $n-1$, and Twin prime numbers as $n-1$ and $n+1$ which satisfy $\sqrt{D(f_{1,n}(x))}\notin\mathbb{Z}$.

\begin{coro}
\upshape
Let $n=2^{p}$ such that $n-1$ is a prime number. Then $f_{1,n}(x)$ is irreducible over $\mathbb{Z}$ and the Galois group is Symmetric group $S_{n-1}$.
\end{coro}

\begin{proof}
$f_{1,n}(x)$ is irreducible by P2. The discriminant $D(f_{1,n}(x))<0$ from P3. Because $2^{p}-1$ is a prime number, by Theorem 1.1, the Galois group of $f_{1,n}(x)$ is Symmetric group $S_{n-1}$.
\end{proof}

\begin{coro}
\upshape
Let $n-1$ and $n+1$ be prime numbers. Then $f_{1,n}(x)$ is irreducible over $\mathbb{Z}$, and if $\sqrt{D(f_{1,n}(x))}\notin\mathbb{Z}$, the Galois group is Symmetric group $S_{n-1}$.
\end{coro}

\begin{proof}
By Proposition 2.3 below, $f_{1,n}(x)$ is irreducible over $\mathbb{Z}$. Hence by Theorem 1.1, if $\sqrt{D(f_{1,n}(x))}\notin\mathbb{Z}$, the Galois group is Symmetric group $S_{n-1}$. As the examples, $n=4, 6, 12, 20$, etc...
\end{proof}

\subsection{The list of Galois groups for $4 \leq n \leq 22$}
By hand calculation and using a computer calculation software "SageMath", we list the Galois groups for for $4 \leq n \leq 22$ below, where $S_{n}$ and $A_{n}$ means $n$-th symmetric group and $n$-th alternative group respectively. In case of only $n = 7$, the Galois group is neither a symmetric nor alternative group. However, the Galois groups are not abelian! Hence we call such Galois extension fields defined by $f_{1,n}(x)$ "non-abelian cyclotomic field over $\mathbb{Q}$ by considering property P1 above.
\begin{table}[htb]
\centering
\caption{List of Galois group}
  \begin{tabular}{|l|c|r|r|} \hline
    $n$ & Galois group & abelian? \\ \hline
    4 & $S_{3}$ & No \\ \hline
    5 & $S_{4}$ & No \\ \hline
    6 & $S_{5}$ & No \\ \hline
    7 & $PGL(2,5)$ & No \\ \hline
    8 & $S_{7}$ & No \\ \hline
    9 & $S_{8}$ & No \\ \hline
    10 & $S_{9}$ & No \\ \hline
    11 & $S_{10}$ & No \\ \hline
    12 & $S_{11}$ & No \\ \hline
    13 & $S_{12}$ & No \\ \hline
    14 & $S_{13}$ & No \\ \hline
    15 & $S_{14}$ & No \\ \hline
    16 & $S_{15}$ & No \\ \hline
    17 & $A_{16}$ & No \\ \hline
    18 & $A_{17}$ & No \\ \hline
    19 & $S_{18}$ & No \\ \hline
    20 & $S_{19}$ & No \\ \hline
    21 & $S_{20}$ & No \\ \hline
    22 & $S_{21}$ & No \\ \hline
  \end{tabular}
\end{table}

\subsection{The reason why we studied the polynomials $f_{1,n}(x)$}
As you know, for any prime number $p$, an positive integer $N$ (not divisible by $p$ and larger than $p$) is uniquely expressed as
\begin{equation}
N=a_{0}+a_{1}p+a_{2}p^{2}+\cdots+a_{n}p^{n}
\end{equation}
where $0\leq a_{i}< p$ for $0\leq i\leq n$, $a_{0}$ and $a_{n}$ are not 0. The author thought that {\it if we replace the prime number $p$ with a variable $x$ in equation (1.3) above, then do the roots of the polynomial have any special mathematical meaning?}\\

Let's think the following equation (1.4) by replacing prime number $p$ with a variable $x$ in (1.3) above:
\begin{equation}
N=a_{0}+a_{1}x+a_{2}x^{2}+\cdots+a_{n}x^{n}
\end{equation}
As a matter of convenience, let's transform (1.4) like below:
\begin{equation}
a_{n}x^{n}+a_{n-1}x^{n-1}+\cdots +a_{2}x^{2}+a_{1}x+a_{0}-N=0
\end{equation}
Because we already know that the equation (1.5) has a root $p$, $(x-p)$ divides left-hand side of (1.5). Let's denote the quotient as $f_{p,N}(x)$:
\begin{equation}
f_{p,N}(x)=(a_{n}x^{n}+a_{n-1}x^{n-1}+\cdots +a_{2}x^{2}+a_{1}x+a_{0}-N)/(x-p)
\end{equation}
This polynomial $f_{p,N}(x)$ is an element of $\mathbb{Z}[x]$, and $f_{1,n}(x)$ of (1.1) is a special case of $f_{p,N}(x)$.
\\

We have defined $f_{p,N}(x)$ for prime number $p$. However we can easily extend the definition for any positive integer $m$ by using $m\mathchar`-adic$ expansion of $N$ by assuming that $m$ and $N$ are relatively prime numbers. As for the special case ``$m=1$'', because $N=1^{1}+1^{2}+\cdots+1^{N}$, we obtain the following equation by replacing ``1'' with a variable $x$.
\begin{equation}
N=x+x^{2}+\cdots+x^{N}.
\end{equation}
By transforming (1.7),
\\

$x^{N}+x^{N-1}+x^{N-2}+\cdots +x^{2}+x-N$
\begin{equation}
=(x-1)(x^{N-1}+2x^{N-2}+\cdots +(N-2)x^{2}+(N-1)x+N).
\end{equation}
Hence we obtain
\begin{equation}
f_{1,N}(x)=x^{N-1}+2x^{N-2}+\cdots +(N-2)x^{2}+(N-1)x+N.
\end{equation}\\
These are the reason why we started to study $f_{1,n}(x)$. Initially, we studied the range of roots of $f_{p,N}(x)$ on complex plain for prime numbers $p$. However, we thought that the studying of the Galois group would be interesting. 

\section*{Acknowledgement}
The author expresses his thanks to his wife, son and daughter who gave him time to study this theme and to write this paper. Moreover, he expresses his thanks to his supervisor who taught him amusingness of number theory when  the author was a graduate student. Thank you very much.
Further, the author was impressed by a bibliography $\cite{Ihara}$. It is a teaching certificate written in Japanese by Professor Ihara for young researcher of mathematics. The author feels happy because he can read the great book in his mother tongue Japanese.


\section{Some Properties of $f_{1,n}(x)$ and $f_{p,N}(x)$}
We discuss the distribution of roots of $f_{1,n}(x)$ and $f_{p,N}(x)$, and the irreducibility in this section.

\begin{prop}
\upshape
If $\alpha$ is a root of $f_{1,n}(x)$, then $1\leq |\alpha|<(n+1)^{2/n}$.
\end{prop}

\begin{proof}
By considering $g_{1,n}(x)=(x-1)f_{1,n}(x)$, we have
\begin{equation}
g_{1,n}(\alpha)=(\alpha-1)f_{1,n}(\alpha)=\alpha^{n} + \alpha^{n-1} + \alpha^{n-2} + \cdots + \alpha - n=0.
\end{equation}
By transposition $-n$ of (2.1) to right-hand side,
\begin{equation}
\alpha^{n} + \alpha^{n-1} + \alpha^{n-2} + \cdots + \alpha=n.
\end{equation}
If $|\alpha|<1$, then $n=|\alpha^{n} + \alpha^{n-1} + \alpha^{n-2} + \cdots + \alpha|\leq |\alpha|^{n}+|\alpha|^{n-2} + \cdots + |\alpha|<1+1+\cdots+1=n$. Because this is inconsitent, $|\alpha|\geq 1$. Next, first let's show that $|\alpha|<n+1$ (the following proof quotes from  \cite{Takagi}). By transforming (2.2) above, we have
\begin{equation}
\alpha^{n} =n-(\alpha^{n-1} + \alpha^{n-2} + \cdots + \alpha),
\end{equation}
and by dividing the both side by $\alpha^{n-1}$,
\begin{equation}
\alpha =-1-\frac{1}{\alpha} - \frac{1}{\alpha^{2}} + \cdots + -\frac{1}{\alpha^{n-2}}+\frac{n}{\alpha^{n-1}}.
\end{equation}
By taking the absolute values of both sides of (2.4),
\begin{equation}
|\alpha| \leq n(1+\frac{1}{|\alpha|} + \frac{1}{|\alpha|^{2}} + \cdots + \frac{1}{|\alpha|^{n-2}}+\frac{1}{|\alpha|^{n-1}}),
\end{equation}
Further by considering the infinite series of the right-hand side,
\begin{equation}
|\alpha| < n(\frac{|\alpha|}{|\alpha|-1}),
\end{equation}
Hence we obtain $|\alpha|<n+1$. By using this fact, we can show $|\alpha|<(n+1)^{2/n}$.\\
By (2.2) above,
\begin{equation}
n=|\alpha||\alpha^{n-1} + \alpha^{n-2} + \cdots + \alpha+1|\geq |\alpha^{n-1} + \alpha^{n-2} + \cdots + \alpha+1|.
\end{equation}
The right side is $|(\alpha^{n}-1)/(\alpha-1)|$. Because of the triangle inequality, $|\alpha-1|\leq|\alpha|+1<n+2$. Hence we have
\begin{equation}
n\geq|\frac{\alpha^{n}-1}{\alpha-1}|>\frac{|\alpha|^{n}-1}{n+2}
\end{equation}
This leads the required inequality $|\alpha|<(n+1)^{2/n}$.
\end{proof}

Although we don't use the next proposition for Theorem 1.1, we show the similar result for $f_{p,N}(x)$ just in case.
\begin{prop}
\upshape
Let $p$ be a prime number and $N$ be a positive integer which is not divisible by $p$. Further assume that the $p\mathchar`-adic$ expanision of $N$ is the same as (1.3). Then if $\alpha_{1}, \alpha_{2}, \cdots, \alpha_{n-1}$ are roots of $f_{p,N}(x)$, $p\leq |\alpha_{i}|< p^{2}$ for any $1\leq i\leq n-1$.
\end{prop}

\begin{proof}
If $|\alpha_{i}|<p$, it leads following contradiction:
\begin{equation}
N=|a_{0}+a_{1}\alpha_{i}+a_{2}\alpha_{i}^{2}+\cdots+a_{n}\alpha_{i}^{n}|\leq a_{0}+a_{1}|\alpha_{i}|+a_{2}|\alpha_{i}|^{2}+\cdots+a_{n}|\alpha_{i}|^{n}<N \nonumber.
\end{equation}
Hence $p\leq |\alpha_{i}|$. On the other hand, due to the relation between roots and coefficients of $f_{p,N}(x)$, we have
\begin{equation}
|\alpha_{1}||\alpha_{2}|\cdots|\alpha_{n-1}|=(N-a_{0})/a_{n}p= (1/a_{n})(a_{n}p^{n-1}+a_{n-1}p^{n-2}+\cdots a_{2}p+a_{1}),
\end{equation}
\begin{equation}
(1/a_{n})(a_{n}p^{n-1}+a_{n-1}p^{n-2}+\cdots a_{2}p+a_{1})\leq a_{n}p^{n-1}+a_{n-1}p^{n-2}+\cdots a_{2}p+a_{1}<p^{n}.
\end{equation}
Hence we obtain $|\alpha_{1}||\alpha_{2}|\cdots|\alpha_{n-1}|<p^{n}$. If $|\alpha_{i}|\geq p^{2}$ for some $i$, then $p^{n}\leq |\alpha_{1}||\alpha_{2}|\cdots|\alpha_{n-1}|$. Because it is inconsistent, we obtain $|\alpha_{i}|<p^{2}$ for any ${i}$.

\end{proof}

Next, we discuss the irreducibility of $f_{1,n}(x)$ over $\mathbb{Z}$ for some $n$.
\begin{prop}
\upshape
If $n+1$ is a prime number, $f_{1,n}(x)$ is irreducible over $\mathbb{Z}$.
\end{prop}

\begin{proof}
We show this proposition by applying Eisenstein's irreducibility criterion. Let's consider $f_{1,n}(x+1)$,
\begin{equation}
f_{1,n}(x+1)=(x+1)^{n-1}+2(x+1)^{n-2}+\cdots +(n-1)(x+1)+n.
\end{equation}
Then, the coefficient of $x^{n-k}$ is
\begin{equation}
\binom{n-1}{k-1}+2\binom{n-2}{k-2}+3\binom{n-3}{k-3}+\cdots +k\binom{n-k}{1}+(k+1)\binom{n-k-1}{0}.
\end{equation}
By the next Lemma 2.1, (2.12) is $\binom{n+1}{k-1}$.  Hence we have
\begin{equation}
f_{1,n}(x+1)=x^{n-1}+\binom{n+1}{1}x^{n-2}+\binom{n+1}{2}x^{n-3}\cdots +\binom{n+1}{n-2}x+n(n+1)/2.
\end{equation}
If $n+1$ is a prime number, because $\binom{n+1}{k}$ is divisible by $n+1$, and $n(n+1)/2$ can be divided by $n+1$ only one time, $f_{1,n}(x+1)$ is irreducible over $\mathbb{Z}$ by Eisenstein's irreducibility criterion. Hence $f_{1,n}(x)$ is irreducible over $\mathbb{Z}$.
\end{proof}

\begin{lemm}
\upshape
\begin{equation}
\binom{n+1}{k-1}=\binom{n-1}{k-1}+2\binom{n-2}{k-2}+3\binom{n-3}{k-3}+\cdots +k\binom{n-k}{1}+(k+1)\binom{n-k-1}{0}.
\end{equation}
\end{lemm}

\begin{proof}
We show this lemma by induction. 
\\

For $n=3$,
\begin{equation}
\binom{3}{1}+2=5=\binom{5}{1},
\end{equation}
\begin{equation}
\binom{3}{2}+2\binom{2}{1}+3=10=\binom{5}{2},
\end{equation}
\begin{equation}
\binom{3}{3}+2\binom{2}{2}+3\binom{1}{1}=10=\binom{5}{3}.
\end{equation}

For $n>3$, because $\binom{n-1}{k-1}=\binom{n-2}{k-2}+\binom{n-2}{k-1}$, (2.14) is extended like below:
\begin{equation}
=\binom{n-2}{k-2}+2\binom{n-3}{k-3}+3\binom{n-4}{k-4}+\cdots +k\binom{n-k-1}{0}
\end{equation}
\begin{equation}
+\binom{n-2}{k-1}+2\binom{n-3}{k-2}+3\binom{n-4}{k-3}+\cdots +k\binom{n-k-1}{1}+k+1.
\end{equation}
By induction, (2.18) is $\binom{n}{k-2}$ and (2.19) is $\binom{n}{k-1}$. Hence we obtain (2.14) because of $\binom{n}{k-2}+\binom{n}{k-1}=\binom{n+1}{k-1}$.

\end{proof}


\begin{prop}
\upshape
If $p$ is a prime number, $f_{1,p}(x)$ is irreducible over $\mathbb{Z}$.
\end{prop}

\begin{proof}
If $f_{1,p}(x)=x^{p-1}+2x^{p-2}+\cdots +(p-1)x+p$ is reducible in $\mathbb{Z}[x]$, we have the following decomposition in $\mathbb{Z}[x]$:
\begin{equation}
f_{1,p}(x)=(x^{m}+a_{m-1}x^{m-1}+\cdots +a_{1}x\pm 1)(x^{l}+b_{l-1}x^{l-1}+\cdots +b_{1}x\pm p).
\end{equation}
The absolute values of roots of the equation $g(x)=(x^{m}+a_{m-1}x^{m-1}+\cdots +a_{1}x\pm 1)$ are 1 because of Proposition 2.1 and the relation between roots and coefficients of $g(x)$. Further, all complex conjugations of the roots are roots of $g(x)=0$ because of $g(x)\in \mathbb{Z}[x]$. Hence by Kronecker's Theorem (see \cite{Kronecker}), the roots of $g(x)$ are roots of $z^{k}=1$ for some $k$.\\
\\
Let $\alpha$ be a root of $g(x)$. Then for some positive integer $d<p-1$, $\alpha^{d}=1$. Hence we have
\begin{equation}
\alpha^{d-1}+\alpha^{d-2}\cdots +\alpha+1=0.
\end{equation}
On the other hand, by the definition of $f_{1,p}(x)$, $\alpha$ satisfies the following equation:
\begin{equation}
\alpha^{p}+\alpha^{p-1}+\cdots + \alpha^{d}+\alpha^{d-1}+\alpha^{d-2}+\cdots +\alpha+1=p+1.
\end{equation}
By combining (2.21) with (2.22), and taking the absolute values, we have
\begin{equation}
p+1=|\alpha^{p}+\alpha^{p-1}+\cdots + \alpha^{d}+\alpha^{d-1}+\alpha^{d-2}+\cdots +\alpha+1|\leq p-d.
\end{equation}
Because this is inconsistent, $f_{1,p}(x)$ is irreducible over $\mathbb{Z}$.

\end{proof}

\begin{prop}
\upshape
Let $p$ be a prime number. Then $f_{1,p^{n}}(x)$ is irreducible over $\mathbb{Z}$ for any integer $n\geq1$.
\end{prop}

\begin{proof}
According to the proof of Proposition 2.4 above, $f_{1,p^{n}}(x)$ doesn't have any root of 1. Hence let's assume that $f_{1,p^{n}}(x)$ is decomposed into the following two factors in $\mathbb{Z}[x]$:
\begin{equation}
f_{1,p^{n}}(x)=(x^{m}+a_{m-1}x^{m-1}+\cdots +a_{1}x \pm p^{r})(x^{l}+b_{l-1}x^{l-1}+\cdots +b_{1}x \pm p^{s}).
\end{equation}
where $r+s=n$, and $r, s >0$.
Then by comparing the coefficients of $x$, $\pm(a_{1}p^{s}+b_{1}p^{r}) = p^{n}-1$. The left-hand side is divisible by $p$, but the right-hand side is not divisible by $p$. This is inconsistent.

\end{proof}


\section{The discriminant of $f_{1,n}(x)$}

\begin{prop}
The discriminant of $f_{1,n}(x)$ is $(-1)^{\frac{(n+2)(n-1)}{2}}2n^{n-3}(n+1)^{n-2}$.
\end{prop}

\begin{proof}
To calculate the discriminant $D(f_{1,n}(x))$, we use $g_{1,n}(x)=(x-1)f_{1,n}(x)$. By the definition of the discriminant, $D(g_{1,n}(x))=\prod_{i=1}^{n-1}(1-\alpha_{i})^{2}D(f_{1,n}(x))$, where $\alpha_{1}, \cdots, \alpha_{n-1}$ are roots of $f_{1,n}(x)$. On the other hand, by the definition of $f_{1,n}(x)$, $\prod_{i=1}^{n-1}(1-\alpha_{i})=f_{1,n}(1)=n(n+1)/2$. Hence we obtain
\begin{equation}
D(g_{1,n}(x))=\{n(n+1)/2\}^{2}D(f_{1,n}(x)).
\end{equation}
We calculate $D(g_{1,n}(x))$ by using Sylvester's determinant Theorem (see \cite{LinearAlgebra}): 
\begin{equation}
D(g_{1,n}(x))=(-1)^{n(n+1)/2}R(g, g')
\end{equation}
where, for simplicity, we denote $g=g_{1,n}(x)$ here.
\\
By the definition, 
\begin{equation}
g_{1,n}(x)=x^{n}+x^{n-1}+\cdots +x-n,
\end{equation}
hence we obtain $R(g, g')=$
\begin{equation}
\left|
\begin{array}{cccccccccccc}
1 & 1 & 1 & \cdots & 1 & 1 & -n & 0 & 0 & 0 & \cdots & 0 \\
0 & 1 & 1 & \cdots & 1 & 1 & 1 & -n & 0 & 0 & \cdots & 0 \\
0 & 0 & 1 & \cdots & 1 & 1 & 1 & 1 & -n & 0 & \cdots & 0 \\
\cdots & \cdots & \cdots & \cdots & \cdots & \cdots & \cdots & \cdots & \cdots & \cdots & \cdots \\
0 & 0 & 0 & \cdots & 1 & 1 & 1 & 1 & 1 & 1 & \cdots & -n \\
n & n-1 & n-2 & \cdots & 2 & 1 & 0 & 0 & 0 & 0 & \cdots & 0 \\
0 & n & n-1 & \cdots & 3 & 2 & 1 & 0 & 0 & 0 & \cdots & 0 \\
0 & 0 & n & \cdots & 4 & 3 & 2 & 1 & 0 & 0 & \cdots & 0 \\
\cdots & \cdots & \cdots & \cdots & \cdots & \cdots & \cdots & \cdots & \cdots & \cdots & \cdots \\
0 & 0 & 0 & \cdots & 0 & n & n-1 & n-2 & n-3 & n-4 & \cdots & 1 \\
\end{array}
\right|
\end{equation}
\\
Please note that this is a determinant of $(2n-1) \times (2n-1)$ matrix.
\\
Next, for $1\leq i \leq n-1$, by subtracting ``$i$-th row''$\times n$ from $(n-1+i)$-th row in (3.4) above, we have
\begin{equation}
\left|
\begin{array}{ccccccccccccc}
1 & 1 & 1 & 1 & \cdots & 1 & 1 & -n & 0 & 0 & 0 & \cdots & 0 \\
0 & 1 & 1 & 1 & \cdots & 1 & 1 & 1 & -n & 0 & 0 & \cdots & 0 \\
0 & 0 & 1 & 1 & \cdots & 1 & 1 & 1 & 1 & -n & 0 & \cdots & 0 \\
\cdots & \cdots & \cdots & \cdots & \cdots & \cdots & \cdots & \cdots & \cdots & \cdots & \cdots & \cdots \\
0 & 0 & 0 & 0 & \cdots & 1 & 1 & 1 & 1 & 1 & 1 & \cdots & -n \\
0 & -1 & -2 & -3 & \cdots & 2-n & 1-n & -n^{2} & 0 & 0 & 0 & \cdots & 0 \\
0 & 0 & -1 & -2 & \cdots & 3-n & 2-n & 1-n & -n^{2} & 0 & 0 & \cdots & 0 \\
0 & 0 & 0 & 0 & \cdots & 4-n & 3-n & 2-n & 1-n & -n^{2} & 0 & \cdots & 0 \\
\cdots & \cdots & \cdots & \cdots & \cdots & \cdots & \cdots & \cdots & \cdots & \cdots & \cdots & \cdots \\
0 & 0 & 0 & 0 & \cdots & 0 & -1 & -2 & -3 & -4 & -5 & \cdots & n^{2} \\
0 & 0 & 0 & 0 & \cdots & 0 & n & n-1 & n-2 & n-3 & n-4 & \cdots & 1 \\
\end{array}
\right|
\end{equation}
\\
For $n\leq i \leq 2n-2$, by adding summation of ``{\it from $(i-n+2)\mathchar`-$th row to $(n-1)\mathchar`-$th row}'' to $i$-th row, (3.5) is equal to the following determinant of $n\times n$ matrix:
\begin{equation}
\left|
\begin{array}{ccccccc}
-1 & n^{2}+n-2 & -3 & -4 & \cdots & -n+1 & -n \\
-1 & -2 & n^{2}+n-3 & -4 & \cdots & -n+1 & -n \\
-1 & -2 & -3 &  n^{2}+n-4 & \cdots & -n+1 & -n \\
\cdots & \cdots & \cdots & \cdots & \cdots & \cdots & \cdots \\
-1 & -2 & -3 &  -4 & \cdots & n^{2}+1 & -n \\
-1 & -2 & -3 &  -4 & \cdots & 1-n & n^{2} \\
n & n-1 & n-2 & n-3 & \cdots & 2 & 1 \\
\end{array}
\right|
\end{equation}
By adding ``$(n-1)$-th row''$\times n$ to $n$-th row, it is equal to the following:
\begin{equation}
(n+1)\times\left|
\begin{array}{ccccccc}
-1 & n^{2}+n-2 & -3 & -4 & \cdots & -n+1 & -n \\
-1 & -2 & n^{2}+n-3 & -4 & \cdots & -n+1 & -n \\
-1 & -2 & -3 &  n^{2}+n-4 & \cdots & -n+1 & -n \\
\cdots & \cdots & \cdots & \cdots & \cdots & \cdots & \cdots \\
-1 & -2 & -3 &  -4 & \cdots & n^{2}+1 & -n \\
-1 & -2 & -3 &  -4 & \cdots & 1-n & n^{2} \\
0 & -1 & -2 & -3 & \cdots & -(n-2) & n^{2}-n+1 \\
\end{array}
\right|
\end{equation}
Further, for $1\leq i \leq n-2$, by subtracting $i$-row from $(i+1)$-row, we obtain
\begin{equation}
(n+1)\times\left|
\begin{array}{ccccccc}
-1 & n^{2}+n-2 & -3 & -4 & \cdots & -n+1 & -n \\
0 & -n^{2}-n & n^{2}+n & 0 & \cdots & 0 & 0 \\
0 & 0 & -n^{2}-n & n^{2}+n & \cdots & 0 & 0 \\
0 & 0 & 0 & -n^{2}-n & \cdots & 0 & 0 \\
\cdots & \cdots & \cdots & \cdots & \cdots & \cdots & \cdots \\
0 & 0 & 0 & 0 & \cdots & -n^{2}-n & n^{2}+n \\
0 & -1 & -2 & -3 & \cdots & -(n-2) & n^{2}-n+1 \\
\end{array}
\right|
\end{equation}
Hence we obtain
\begin{equation}
=n^{n-2}(n+1)^{n-1}\times\left|
\begin{array}{ccccccc}
-1 & n^{2}+n-2 & -3 & -4 & \cdots & -n+1 & -n \\
0 & -1 & 1 & 0 & \cdots & 0 & 0 \\
0 & 0 & -1 & 1 & \cdots & 0 & 0 \\
0 & 0 & 0 & -1 & \cdots & 0 & 0 \\
\cdots & \cdots & \cdots & \cdots & \cdots & \cdots & \cdots \\
0 & 0 & 0 & 0 & \cdots & -1 & 1 \\
0 & -1 & -2 & -3 & \cdots & -(n-2) & n^{2}-n+1 \\
\end{array}
\right|
\end{equation}
\begin{equation}
=-n^{n-2}(n+1)^{n-1}\times\left|
\begin{array}{cccccc}
-1 & 1 & 0 & \cdots & 0 & 0 \\
0 & -1 & 1 & \cdots & 0 & 0 \\
0 & 0 & -1 & \cdots & 0 & 0 \\
\cdots & \cdots & \cdots & \cdots & \cdots & \cdots \\
0 & 0 & 0 & \cdots & -1 & 1 \\
-1 & -2 & -3 & \cdots & -(n-2) & n^{2}-n+1 \\
\end{array}
\right|
\end{equation}
\begin{equation}
=(-1)^{n-1}n^{n-2}(n+1)^{n-1}\times\left|
\begin{array}{cccccc}
-1 & -2 & -3 & \cdots & -(n-2) & n^{2}-n+1 \\
-1 & 1 & 0 & \cdots & 0 & 0 \\
0 & -1 & 1 & \cdots & 0 & 0 \\
0 & 0 & -1 & \cdots & 0 & 0 \\
\cdots & \cdots & \cdots & \cdots & \cdots & \cdots \\
0 & 0 & 0 & \cdots & -1 & 1 \\
\end{array}
\right|
\end{equation}
In (3.11), add $(n-1)$-th column to $(n-2)$-th column, add $(n-2)$-th column to $(n-3)$-th column, $\cdots$ by repeating this operation, we obtain
\begin{equation}
=(-1)^{n-1}n^{n-2}(n+1)^{n-1}\times\left|
\begin{array}{cccccc}
n(n+1)/2 & 0 & 0 & \cdots & 0 & 0 \\
0 & 1 & 0 & \cdots & 0 & 0 \\
0 & 0 & 1 & \cdots & 0 & 0 \\
0 & 0 & 0 & \cdots & 0 & 0 \\
\cdots & \cdots & \cdots & \cdots & \cdots & \cdots \\
0 & 0 & 0 & \cdots & 0 & 1 \\
\end{array}
\right|
\end{equation}

$=(-1)^{n-1}n^{n-1}(n+1)^{n}/2$.
\\

By combining this result above with (3.1) and (3.2), we obtain $D(f_{1,n}(x))=(-1)^{\frac{(n-1)(n+2)}{2}}2n^{n-3}(n+1)^{n-2}$.
\end{proof}

\begin{coro}
\upshape
For $n>3$, let $K_{n}$ be the minimal splitting field of $f_{1,n}(x)$ over $\mathbb{Q}$. Then $K_{n}$ has the following quadratic field $F=\mathbb{Q}(\sqrt{D(f_{1,n}(x))})$ as a sub field:
\begin{eqnarray}
F=\left\{ \begin{array}{ll}
\mathbb{Q}(\sqrt{-2l}) & (n=4l) \\
\mathbb{Q}(\sqrt{2l+1})  & (n=4l+1) \\
\mathbb{Q}(\sqrt{2l+1})  & (n=4l+2) \\
\mathbb{Q}(\sqrt{-2(l+1)})  & (n=4l+3) \\
\end{array} \right.
\end{eqnarray}
where we assume that $2l+1$ is not a square of rational integer for $n\equiv 1$ or $2$ (mod $4$).
\end{coro}

\begin{proof}
The comes from the following calculations of the discriminant $D(f_{1,n}(x))$.
\begin{eqnarray}
D(f_{1,n}(x))=\left\{ \begin{array}{ll}
-2l\times 4\times (4l)^{4l-4}\times (4l+1)^{4l-2} & (n=4l) \\
\{2^{2l}\times (4l+1)^{l-1}\times(2l+1)^{2l}\}^{2}\times(2l+1)^{-1}  & (n=4l+1) \\
\{2(2l+1)(4l+3)\}^{4l}(2l+1)^{-1}  & (n=4l+2) \\
-2(l+1)\times 4^{4l+1}(4l+3)^{4l}(l+1)^{4l}  & (n=4l+3) \\
\end{array} \right.
\end{eqnarray}
\end{proof}

\section{The generalization and a conjecture}

In this section, we generalize the polynomial $f_{1,n}(x)$ over $\mathbb{Z}$. In order to generalize it, first we study the case of $n=\infty$.
At a glance, it seems that we can't extend the definition of $f_{1,n}(x)$ for $n=\infty$. However, we can overcome this by considering a transformation $f_{1,n}(1/x)$.
\begin{equation}
f_{1,n}(1/x)=1/x^{n-1}+2/x^{n-2}+3/x^{n-3}+\cdots +(n-1)/x+n.
\end{equation}
By multiplying $x^{n-1}$ for (5.1), we define the ``inverse'' $\bar{f}_{1,n}(x)$ of $f_{1,n}(x)$,
\begin{equation}
\bar{f}_{1,n}(x)= x^{n-1}f_{1,n}(1/x)=1+2x+3x^{2}+\cdots +(n-1)x^{n-2}+nx^{n-1}.
\end{equation}
\\
By the definition of $\bar{f}_{1,n}(x)$, we can extend the definition in formal power series ring $\mathbb{Z}[[x]]$ for $n=\infty$:
\begin{equation}
\bar{f}_{1,\infty}(x)=1+2x+3x^{2}+\cdots +(n-1)x^{n-2}+nx^{n-1}+\cdots.
\end{equation}
By easy calculation, $\bar{f}_{1,\infty}(x)$ is decomposed into square of an element of $\mathbb{Z}[[x]]$:
\begin{equation}
\bar{f}_{1,\infty}(x)=(1+x+x^{2}+\cdots +x^{n}+\cdots)^{2}.
\end{equation}
\\
Hence we can generalize $f_{1,n}(x)$ by taking $m$-th power of the following infinite series $h(x)$:
\begin{equation}
h(x)=1+x+x^{2}+\cdots +x^{n}+\cdots.
\end{equation}
\begin{equation}
h(x)^{m} = 1+ mx + \binom{m+1}{2}x^{2} + \binom{m+2}{3}x^{3} + \cdots + \binom{m+n-1}{n}x^{n} + \cdots
\end{equation}

\begin{defi}
\upshape
For integers $n\geq 3$ and $m\geq 2$, we define a polynomial $_{m}f_{1,n}(x)$ over $\mathbb{Z}$.
\begin{eqnarray}
_{m}f_{1,n}(x) = \sum^{n-2}_{i=-1}\binom{m+i}{i+1}x^{n-2-i},
\end{eqnarray}
i.e., $_{m}f_{1,n}(x) = x^{n-1} + \binom{m}{1}x^{n-2} + \binom{m+1}{2}x^{n-3} + \cdots \binom{m+i}{i+1}x^{n-2-i} + \cdots + \binom{m+n-3}{n-2}x + \binom{m+n-2}{n-1}.$
\end{defi}

\begin{exam}
\upshape
For $n=3$, $f(x) = _{m}f_{1,3}(x) = x^2 + mx + \frac{(m+1)m}{2}$. Because the discriminant of $f(x)$ is $-m(m+2) < 0$, $f(x)$ is irreducible over $\mathbb{Z}$.
\end{exam}

\begin{exam}
\upshape
For $n=4$, $f(x) = _{m}f_{1,4}(x) = x^3 + mx^2 + \frac{(m+1)m}{2}x + \frac{(m+2)(m+1)m}{6}$. Because $f(x)$ has one real root and two complex number roots, it is irreducible over $\mathbb{Z}$. Further, because the discriminant of $f(x)$ is $-\frac{m^{2}(m+1)(m+2)(m+3)^{2}}{6}$, its Galois group of $f(x)$ over $\mathbb{Q}$ is $S_{3}$.
\end{exam}

\subsection{The irreducibility conjecture}
By considering the result in $\S$2, two examples above and the definition of $_{m}f_{1,n}(x)$, we have the following conjecture.

\begin{conji}
\upshape
For any integers $n\geq3$ and $m\geq 2$, $_{m}f_{1,n}(x)$ is irreducible over $\mathbb{Z}$. 
\end{conji}

\section{Class Field Theory in case of $n = 4$ and $m=2$}
In this section, we show an example of non-abelian class field theory for $f(x) = f_{1,4}(x) = x^3 + 2x^2 + 3x +4$. Recall that the discriminant of $f(x)$ is $-2^{3}5^{2}$ and its Galois group $G_{4}$ of the minimal splitting field $L$ of $f(x)$ over $\mathbb{Q}$ is $S_{3}$. Furthermore, the minimal splitting field $L$ has only one quadratic sub field $K = \mathbb{Q}(\sqrt{-2})$.

\begin{thm} 
\upshape
A prime number $p$ splits completely in $L$ if and only if $p=x^{2} + 2y^{2}$ for some integers $x, y$ which satisfies $xy\equiv 0$ (mod 15). 
\\
\\
In other word, $p$ splits completely in $L$ if and only if $a(p) = 2$ in the following infinite series:
\begin{equation}
\theta(\tau) = \sum_{\mathfrak{a}\in I_{\mathfrak{m}}}\chi(\mathfrak{a})q^{N(\mathfrak{a})} = \sum^{\infty}_{n=1}a(n)q^{n}
\end{equation}
\\
\\
where, $\mathfrak{m} = (5\sqrt{-2})$ is an ideal of integer ring $\mathcal{O}_{K} =\mathbb{Z}[\sqrt{-2}]$, $I_{\mathfrak{m}} = \{(\alpha)|\alpha \in K^{*}, (\alpha, \mathfrak{m}) = 1\}$, and $H = \{(\alpha)|\alpha \in K^{*}, \alpha \equiv l$ (mod $\mathfrak{m}$) for some $l\in \mathbb{Z}$ $\}\cup \{a + b\sqrt{-2}|a, b \in \mathbb{Z}$ and $ab\sqrt{-2}\equiv 0$ (mod $\mathfrak{m}$)$\}$. $\chi : I_{\mathfrak{m}}/H \rightarrow \mathbb{C}^{*}$ is a Hecke character defined by $\mathfrak{p}_{3} = (1+\sqrt{-2})\mapsto \omega = (1+\sqrt{-3})/2)$, and $H$ maps to 1.
\end{thm}

\begin{proof}
The core is the following two simple results:
\\
(1) If $p = x^2 + 2y^2$ for some integers $x,y$, then $xy\equiv 0$ (mod 3).
\\
Because $p^{2} \equiv x^{4} +y^{4} + x^{2}y^{2} \equiv 1$ (mod 3), if $xy\not\equiv 0$ (mod 3), it indicates $3\equiv1$ (mod 3). This is a contradiction. 
\\
(2) If we put $(a+b\sqrt{-2})^{3} = X+\sqrt{-2}Y$, then $XY\equiv 0$ (mod 5).
\\
Actually, $(a+b\sqrt{-2})^{3} = a(a^{2}-6b^{2}) + b(3a^{2}-2b^{2})\sqrt{-2}$. Hence $XY = ab(a^{2}-6b^{2})(3a^{2}-2b^{2}) \equiv 3ab(a^{4}-b^{4})\equiv 0$ (mod 5). As an aside, the author thought he met God of Numbers in these calculations.
\\
\\
Because it is easy to show that the generator of $I_{\mathfrak{m}}/H$ is a class of an ideal $\mathfrak{p}_{3}=(1+\sqrt{-2})$, and $I_{\mathfrak{m}}/H$ is isomorphic to $\mathbb{Z}/3\mathbb{Z}$, leave it to the readers.
\end{proof}

\begin{coro}
\upshape
The infinite series (5.1) cannot be expressed by a product of two Dedekind's $\eta(\tau)$ functions $\eta(a\tau)\eta(b\tau)$ ($a,b\geq 1$), where $\eta(\tau) = e^{2\pi i\tau/24}\prod_{n=1}^{\infty}(1-e^{2\pi i\tau n}) = q^{1/24}\sum_{n\in \mathbb{Z}}(-1)^{n}q^{n(3n-1)}$ ($q=e^{2\pi i\tau}$ and $Im(\tau)>0$).
\end{coro}

\begin{proof}
If $\theta(\tau) = \eta(a\tau)\eta(b\tau)$ for some $a, b \geq 1$, the following equation must state by considering index of $q$:
\begin{equation}
a(6n-1)^{2} + b(6m-1)^{2} = 12p
\end{equation}
where $p$ is a prime number which splits completely in $L$ and $n, m\in \mathbb{Z}$. However, it is not possible. The reader can confirm this easily by using spread sheet.
\end{proof}


--- The contact information of the author---
\\
email: canjinana@me.com
\\

\end{document}